\documentclass[journal]{IEEEtran}
\usepackage{cite}

\ifCLASSINFOpdf
\else
\fi

\usepackage{graphicx} 
\usepackage{epsfig} 
\usepackage{times} 
\usepackage{amsmath} 
\usepackage{amssymb} 
\usepackage{multicol}
\usepackage{url}
\newtheorem{theorem}{Theorem}

\newtheorem{assumption}[theorem]{Assumption}
\usepackage{color}
\usepackage{mathrsfs}
\usepackage{multirow,minipage}
\usepackage{caption}
\usepackage{subcaption}
\usepackage{wrapfig}

\newtheorem{definition}[theorem]{Definition}

\newtheorem{lemma}[theorem]{Lemma}

{ 
\newtheorem{remark}[theorem]{Remark}

}
\newenvironment{proof}[1][Proof]{\textbf{#1.} }{\ \hspace*{\fill} \rule{0.5em}{0.5em}}

\newcommand{\bi}{\begin{itemize}}
\newcommand{\ei}{\end{itemize}}
\newcommand{\bd}{\begin{displaymath}}
\newcommand{\ed}{\end{displaymath}}
\newcommand{\be}{\begin{eqnarray*}}
\newcommand{\ee}{\end{eqnarray*}}

\begin{document}
%
\title{Fragility of Decentralized Load-Side Frequency Control in Stochastic Environment}

%
%
%


\author{ Sai Pushpak and Umesh Vaidya
\thanks{Financial support from the National Science Foundation grant CNS-1329915 and ECCS-1150405 is gratefully acknowledged. U. Vaidya is with the Department of Electrical \& Computer Engineering,
Iowa State University, Ames, IA 50011.}}
\maketitle
\begin{abstract}
In this paper, we demonstrate the fragility of decentralized load-side frequency algorithms proposed in \cite{Low_Zhao_control} against stochastic parametric uncertainty in power network model. The stochastic parametric uncertainty is motivated through the presence of renewable energy resources in power system model. 
We show that relatively small variance value of the parametric uncertainty affecting the system bus voltages cause the decentralized load-side frequency regulation algorithm to become stochastically unstable.  The critical variance value of the stochastic bus voltages above which the decentralized control algorithm become mean square unstable is computed using an analytical framework developed in \cite{Sai_arxiv,pushpak2015stability}. Furthermore, the critical variance value is shown to decrease with the increase in the cost of the controllable loads and with the increase in penetration of renewable energy resources. 
Finally, simulation results on IEEE 68 bus system are presented to verify the main findings of the paper. 
\end{abstract}

\section{Introduction}

In this paper, we show that the decentralized load-side frequency control algorithm proposed in \cite{Low_Zhao_control} is very fragile to parametric uncertainty in power network model. One of the essential components of the smart grid vision is the active participation of loads for improved operation and performance of network power system at different time scales \cite{doe_demandresponse}. The technology is maturing to the point where the smart grid vision can be realized for actively controlling the loads to absorb not only the long-term variability or uncertainty from renewable power generation but also short term fluctuations. The application of load-side control for frequency regulation falls into the latter category. With the potential benefits of active load control, there are increased research efforts towards the development of systematic analytical methods and optimization-based tools for distributed load control. Most of the literature on this topic primarily focus on stability properties of control algorithms developed for load-side frequency regulation    \cite{trudnowski2006power,short2007stabilization,molina2011decentralized,andreasson2013distributed,Namerikawa_distributed,Low_Zhao_distributed}. In particular, \cite{Low_Zhao_control} proves the asymptotic stability of primal-dual gradient system leading to the decentralized algorithm for load-side frequency control. However, the important issue related to the performance of these algorithms is not addressed primarily in the presence of parametric uncertainty in power network model. 

For the successful implementation of the various developed algorithms, it is important to analyze the performance of these algorithms against both parametric and additive sources of uncertainties. In this paper, we extended the primal-dual gradient system model of power swing equation with controllable loads developed in \cite{Low_Zhao_control} to incorporate the stochastic parametric uncertainty. There are various sources of parametric uncertainty in power system model. In this paper, we argue that renewable energy resources in the form of wind and solar are the potential source of parametric uncertainty in the network power system, where stochasticity in the availability of renewable energy resources will lead to stochastic bus voltages. We develop stochastic power network model with the additive as well as the multiplicative source of uncertainty. The stochastic notion of mean square stability is used to analyze the stability of network power system using stochastic stability framework for the continuous-time system developed in \cite{Sai_arxiv,pushpak2015stability,sai_acc2016}. The continuous-time stochastic stability framework discussed in this paper is  motivated from stochastic stability analysis and control results developed for linear and nonlinear systems in    \cite{diwadkar2015control,diwadkar2014stabilization,vaidya2010limitations,elia2005remote}. The developed framework is used to determine the critical variance, $\sigma_*^2$, of parametric stochastic uncertainty above which, the system is mean square unstable. We show that decentralized load-side frequency regulation algorithm developed in \cite{Low_Zhao_control} is extremely fragile to stochastic fluctuations in bus voltages. In particular, we show that, with an increase in the cost of the controllable loads, the value of critical variance, $\sigma^2_*$, above which the system is unstable decreases. Furthermore, $\sigma_*^2$ value also decreases with the increase in the penetration of the renewable energy resources in the power network.

The organization of the paper is as follows. In section \ref{sec_stochastic_model}, we provide a framework for incorporating stochastic parametric uncertainty in dynamic power system model employed for solving the load-side frequency regulation problem. Results on analyzing the mean square stability of stochastic power network model are discussed in section \ref{sec_stochastic_stability}. Simulation results on IEEE 68 bus system are presented in section \ref{sec_68bus} followed by conclusion in section \ref{sec_conclusion}.

\section{Load-side frequency control model with stochasticity} 
\label{sec_stochastic_model}

In this section, we develop an integrated load-side frequency control model with stochastic uncertainty. There are various sources of stochastic uncertainty in a power network and renewables energy sources such as solar and wind energy forms the major contributors to the uncertainty. We first discuss briefly the deterministic load-side frequency control model as developed in \cite{Low_Zhao_control}. We refer the readers to \cite{Low_Zhao_control} for more detailed discussion on this model. The basic idea behind the load-side frequency control is to control the loads, so that, the system-wide frequency can be regulated following a small change in power injection at one of the system bus. The load control comes at a cost and is measured by the aggregate disutility of the loads. The objective is to regulate the system frequency while minimizing the aggregate disutility of the loads. 

To set up the problem, consider a power network as a graph network with generator and load buses as nodes and transmission lines as edges. Let the sets $\mathcal{G}, \mathcal{L}$, and $\mathcal{E}$ respectively, denote the set of generator nodes, set of load nodes, and the set containing network edges. The cardinalities of the sets $\mathcal{G}, \mathcal{L}$, and $\mathcal{E}$ are denoted by $n_g, n_l$, and $p$ respectively. Let $\mathcal{N}$ denote the set of generator and load nodes and its cardinality is given by $n$, where $n = n_g + n_l$. The dynamic network model for regulating system frequency using load can be written as follows (we refer the readers to \cite{Low_Zhao_control} for various assumption leading to this model). 
\begin{eqnarray}
\dot \omega_j&=&-\frac{1}{M_j} (\hat d_j +d_j-P_j^m+P_j^{out}-P_j^{in}),\;\forall j\in {\cal G}\nonumber\\
0&=&\hat d_j +d_j-P_j^m+P_j^{out}-P_j^{in},\;\;\forall j\in {\cal L}\nonumber\\
\dot P_{ij}&=&W_{ij}(\omega_i-\omega_j),\;\;\forall (i,j)\in {\cal E}\label{system}
\end{eqnarray}
where $\omega_j$ is  the frequency deviation at $j^{th}$ bus, $W_{ij} :=  3 \frac{\vert V_i \vert \vert V_j \vert}{X_{ij}} \cos(\theta_i^0 - \theta_j^0),$
where $V_i, \theta_i^0$ are the voltage and nominal phase angle at bus $i$, and $X_{ij}$ is the reactance between the buses $i$ and $j$. Three types of loads are distinguished in the above model namely, frequency-sensitive, frequency-insensitive but controllable, and uncontrollable loads. The quantity, $\hat d_j$  models the frequency-sensitive load and is assumed to be of the form $\hat d_j=D_j \omega_j$, i.e., it responds linearly to  frequency deviation. Further, $P_j^m$ incorporates the part of load which is frequency-insensitive and uncontrollable and $d_j$ models the load which is frequency-insensitive but controllable. 

The objective is to design a feedback controller $d_j(\omega(t),P(t))$ for the controllable loads, so that, frequency can be regulated following disturbance, i.e., the system (\ref{system}) is globally asymptotically stable. In \cite{Low_Zhao_control}, an alternate optimization-based approach is proposed for adjusting the controllable load, $d_j$. The design of feedback controller, $d_j(\omega(t),P(t))$, is posed as an optimal load control (OLC) problem and the feedback controller is derived as a distributed algorithm to solve the OLC. The optimization problem for OLC is formulated as follows.
\begin{align}
\min_{\underline{d}\leq d \leq \overline{d},\hat{d}} \quad & {\sum_{j \in \mathcal{N}}} \left(c_j(d_j) + \frac{1}{2D_j} \hat{d}_j^2\right) \label{lfc_optm_cost}\\
\rm{subject\; to}\quad & \sum_{j \in \mathcal{N}} (d_j + \hat{d_j}) = \sum_{j \in \mathcal{N}} P_j^m \label{lfc_optm_constraints}
\end{align}
where $c_j(d_j)$ is the cost on the controllable load at bus $j$, when it is changed by $d_j$ and $\hat{d}_j := \hat D_j \omega_j$ denotes the frequency deviation, $\omega_j$ of the frequency sensitive load at bus $j$. The change in either generator or load bus $j$ is denoted by $P_j^m$ and the loads always satisfy $-\infty < \underline{d}_j \leq d_j\leq \overline{d}_j < \infty$. 
Furthermore, the cost function $c_j$ at every bus $j$ is assumed to be strictly convex and twice continuously differentiable on $[\underline{d}_j,\; \overline{d}_j]$. A dual to the OLC problem (\ref{lfc_optm_cost})-(\ref{lfc_optm_constraints}) which can be solved or implemented with a distributed architecture is written as 
\begin{align}
\max_{\nu} \quad & \textstyle{\sum_{j \in \mathcal{N}}} \Phi_j(\nu_j)\label{dual_cost} \\
\rm{subject\; to}\quad & \nu_i = \nu_j, \quad \forall (i,j) \in \mathcal{E},\label{dual_constraints}
\end{align}
where $\Phi_j(\nu_j) = c_j(d_j(\nu_j)) - \nu_j d_j(\nu_j) - \frac{1}{2} \hat D_j \nu_j^2 + \nu_j P_j^m$, $d_j(\nu_j) = \left[c_j^{'-1}(\nu_j) \right]^{\overline{d}_j}_{\underline{d}_j}$, and  $c_j^{'}(\nu_j)$ is the derivative of the cost function. Note that $\Phi_j$ is only a function of $\nu_j$, the dual variable, and all $\nu_j$ are constrainted to be equal at optimality. After some change of variables, it can be shown that the primal-dual gradient system corresponding to optimization problem (\ref{dual_cost})-(\ref{dual_constraints}) takes the form as given below \cite{Low_Zhao_control}. 
\begin{align}
 \dot{\omega}_j = & \frac{-1}{M_j} (\hat d_j  +  d_j -  P_j^m +  P_j^{out} -  P_j^{in}), \forall j \in \mathcal{G}, \label{eq_gen1} \\
0 = & \hat d_j  +  d_j -  P_j^m +  P_j^{out} -  P_j^{in},  \forall j \in \mathcal{L}, \label{eq_load} \\
 \dot{P}_{ij} = & W_{ij}  ( \omega_i -  \omega_j), \forall (i,j) \in \mathcal{E}, \label{eq_power} \\
 \hat d_j = & \hat D_j \omega_j, \forall j \in \mathcal{N}, \label{eq_freq_sensitive_load} \\
 d_j = & \left[c_j^{'-1} ( \omega_j) \right]^{\overline{{d}}_j}_{\underline{d}_j}, \; \forall j \in \mathcal{N}. \label{control_law}
\end{align} 

Notice that, the primal-dual gradient system is same as the dynamic network model for load frequency control, i.e., Eq. (\ref{system}) except for the fact that the decentralized feedback control law for the controllable load is obtained as the solution of this optimization problem. This implies that the dynamic power network model is essentially implementing the primal-dual gradient algorithm, where the feedback control law (\ref{control_law}) needs to be implemented at each controllable load for decentralized load-side frequency regulation of power network. In \cite{Low_Zhao_control}, the authors prove the existence of unique equilibrium point for primal-dual gradient system which is globally asymptotically stable. The objective of this paper is to show that, the asymptotically stable property of this equilibrium point is very fragile to the presence of multiplicative stochastic uncertainty in the power network. In the following section, we motivate the stochastic dynamic network model using uncertain and intermittent nature of wind energy. 

\subsection{Stochastic Renewables and Parametric Uncertainty}
In this subsection, we show how the parametric uncertainty enters in the power system dynamics. We motivate the parametric uncertainty in the power system dynamics through the presence of renewables where the intermittent nature of  wind and solar energy resources are modeled as stochastic random variables. The network power system with parametric uncertainty can be modeled as a set of differential algebraic equations (DAEs) written as follows:
\begin{eqnarray}
\dot x&=&f(x,y,\xi)\label{dynamic}\\
0&=&g(x,y,\xi)\label{algebra}
\end{eqnarray}
where, $x$ are the dynamic states corresponding to generator angular velocities, generator excitation voltages, power across the transmission lines, etc., and $y$ are the network as well as generator algebraic states corresponding to the bus voltages, bus angles, currents, etc. and $\xi$ denotes the stochastic parametric uncertainty. In the following discussion we will identify the precise source of parametric uncertainty in both the differential equation and algebraic equation. The network power system model with renewable wind energy resources will consists of conventional synchronous generators as well as doubly fed induction generators (DFIG). We refer the readers to \cite{pulgar2011wind} for more detailed discussion on the deterministic modeling of network power system with renewable wind generation. A zero-axis model \cite{pulgar2011wind} for DFIG obtained by neglecting the dynamics of the stator and rotor flux linkages is given by,   
\begin{eqnarray}
\frac{d\omega_r}{dt} = & \frac{\omega_s}{2H_D} \left(T_{mD} - \bar X_m I_{qs} I_{dr} +\bar X_m I_{ds} I_{qr} \right) \label{dfigd}\label{diff_eq1} 
\end{eqnarray}
where $\omega_r$ is the electrical rotor speed of DFIG, $I_{qr},I_{qs},I_{dr}, I_{ds}$ are the algebraic states of DFIG, and  $T_{mD} = \bar B \omega_b C_p(\lambda, \theta) \frac{v_{wind}^3}{\omega_r}$. The wind speed $v_{wind}$ is intermittent in nature and hence can be modeled as stochastic random variable as follows:
\begin{eqnarray}
v_{wind}^3  = & v_{{wind}_0}^3 + \xi \label{winds}
\end{eqnarray}
where $v_{{wind}_0}$ is the nominal wind speed and $\xi$ is the stochastic uncertainty. Now substituting (\ref{winds}) in the expression of $T_{mD}$ and after substituting $T_{mD}$ in (\ref{dfigd}), we notice that the random variable, $\xi$, multiply system state $\frac{1}{\omega_r}$ and hence enter parametrically in the dynamic state equation of power system. The presence of parametric uncertainty in the algebraic equation of the power system can be explained as follows. The algebraic equation corresponding to DFIG affected by the wind speed is written as follows \cite{pulgar2011wind}
\[0 =  -V_{qr} +K_{P2}[K_{P1}(P_{ref} - P_{gen})+z_1 - I_{qr}]+z_2,\]
where $V_{qr}, I_{qr}$ are the algebraic states of DFIG and $z_1,z_2$ are the dynamic states of speed controller of DFIG. The power reference input is written as 
\[P_{ref} = \omega_{ref} T_{mD}. \]
Again substituting  $T_{mD} = \bar B \omega_b C_p(\lambda, \theta) \frac{1}{\omega_r} (v_{{wind}_0}^3 + \xi)$, we notice that the stochastic parametric uncertainty enters the DAE equations of DFIG.

 In the absence of stochastic uncertainty, i.e., when $\xi = 0$ and under the assumption that $\frac{\partial g}{\partial y}\neq 0$, implicit function theorem can be applied to network algebraic equation (\ref{algebra}) to eliminate the algebraic state $y$ by expressing $y=h(x)$. In the presence of stochastic uncertainty, an argument involving center manifold based reduction for stochastic system and singular perturbation theory for stochastic system \cite{arnold2013random,berglund2003geometric} the algebraic states, $y$ can be expressed as a stochastic function of states $x$ i.e., $y=h(x,\xi)$. Using this in Eq. \eqref{dynamic}, we obtain, 
\[\dot x = f(x,h(x,\xi),\xi)\]
The above system can be linearized at a nominal operating point to obtain a linear system, where the stochastic uncertainty enters the linearized system parametrically. 
%

In the following, we show how the stochastic uncertainty in the algebraic states propagates into the network power system. One of the algebraic states that is of particular interest to us is the bus voltages. It is clear that uncertainty in renewables will cause the voltage to behave randomly. Apart from voltages, there are other network parameters that one can assume to be uncertain and hence modeled as a stochastic random variable. For example, the frequency-sensitive loads can be assumed to be uncertain, i.e., $\hat d_j=(D_o+d\xi D_1)\omega_j$, where $d\xi$ is stochastic process. Note that the uncertainty is assumed to be parametric, where the damping coefficient is changing over time. The loads are constantly turned on and off in the grid and thereby changing the effective damping coefficient of the frequency-sensitive loads. 
Similarly, the frequency-insensitive uncontrollable loads can also be uncertain. However, in this paper, we mainly focus on the bus voltages being uncertain and analyze the impact of stochastic voltage fluctuations on the system stability. Suppose, $p_1 < n_g$ generators in the power network are now replaced with a renewable energy source. As we are modeling the voltages at renewable buses to be stochastic, the voltages at buses connecting the renewables are also stochastic. Let ${\cal S}$ be the set of a pair of buses whose voltages are stochastic and its cardinality be denoted by $s < p$, where $p$ is the number of total links in the network. 
Under the assumption that the nominal voltages are $1$ p.u., stochastic voltage fluctuations are modeled as follows:
\begin{eqnarray}
|V_i||V_j|= 1+\sigma d\xi_{k} \;\;\forall (i,j)\in{\cal S}\label{stoc_voltage}
\end{eqnarray}
where $d\xi_{k}$ is the standard Wiener process and $\sigma$ is the standard deviation assumed to be same for all links. We use unique index $k$ to identify and denote the edge pair $(i,j)\in {\cal S}$ and hence $k=1,\ldots, s$. 

 Furthermore $d\xi_{k}$ is assumed to be independent of $d\xi_\ell$ for $k\neq \ell$. For the simplicity of presentation, we assume that all the links in the networks have the same variance, $\sigma^2$. 
 
 \begin{remark}Notice that instead of assuming individual bus voltages to be random,  we are assuming product of voltages to be random. This is a modeling assumption and is made to avoid technical difficulty that arises while  multiplying two stochastic processes. 
\end{remark} 
 We now make following assumption on the cost function $c_j$.
\begin{assumption} \label{cost_assumption}We assume that the cost function $c_j$ is quadratic and hence of the form $c_j(d_j)=\frac{d_j^2}{2\alpha_j}$ for all $j\in {\cal N}$. Furthermore, we neglect the saturation constraints on the cost function and hence optimal decentralized control law for the controllable load is of the form $d_j=\alpha_j\omega_j$.
\end{assumption}

\section{Stochastic stability of power network}
\label{sec_stochastic_stability}
The deterministic dynamic network model (\ref{eq_gen1})-(\ref{control_law}) can be combined with the stochastic voltage fluctuation model (\ref{stoc_voltage}) to write a power system model with multiplicative stochastic uncertainty. First, we write the deterministic power network model (\ref{eq_gen1})-(\ref{control_law}) using Assumption \ref{cost_assumption} in compact form after eliminating the algebraic equation (\ref{eq_load}) as  follows:
\begin{equation}
\begin{aligned}
\dot{\omega}_G = & - M_G^{-1} (D_G \omega_G + E_G P - P_G^m)  \\
\dot P = & W (E_G^{\top} \omega_G + E_L^{\top} D_L^{-1}(P_L^m-E_L P)) 
\end{aligned}
\label{det_model}
\end{equation}
%
where $\omega_G \in \mathbb{R}^{n_g}$, $\omega_L \in \mathbb{R}^{n_l}$ and $P \in \mathbb{R}^{p}$. 
Observe that $M_G, D_G, D_L$ are diagonal matrices. The weight matrix, $W \in \mathbb{R}^{p \times p}$ is defined as a diagonal matrix with entries $W_{ij}$ for all $(i,j) \in \mathcal{E}$. The matrices $E_G$ and $E_L$ are the incidence matrices corresponding to generator and load buses respectively. 
Next, we incorporate the uncertainty in the above deterministic model. Using Eq. (\ref{stoc_voltage}), we can write the stochastic link weight, $W_{ij}$ as follows
$W_{ij}=3\frac{(1+\sigma d\xi_k)}{X_{ij}}\cos (\theta_i^0-\theta_j^0).$
Define $W_{ij}^0:=3\frac{1}{X_{ij}}\cos (\theta_i^0-\theta_j^0)$, and hence, we have 
\begin{eqnarray}
W_{ij}=W_{ij}^0+\sigma W_{ij}^0  d\xi_k\label{stoc_weight}.
\end{eqnarray}
Substituting (\ref{stoc_weight}) in (\ref{det_model}), we obtain following stochastic power network model with some abuse of notation. 
{\small
\begin{eqnarray}
\dot{\omega}_G &= & - M_G^{-1} (D_G \omega_G + E_G P - P_G^m) \label{stoc_model}\\
\dot P &= & (W^0+\sigma W^0\circ d\xi) (E_G^{\top} \omega_G + E_L^{\top} D_L^{-1}(P_L^m-E_L P))  \nonumber
\end{eqnarray}
}
where $W^0={\rm diag} (W_{ij}^0)$ for $(i,j)\in {\cal E}$, $d\xi$ is a diagonal matrix with zeros and $d\xi_1,\ldots,d\xi_s$. The nonzero entries of $d\xi$ correspond to the links given in set ${\cal S}$. The symbol, $\circ$ denotes element-wise matrix multiplication. To represent the system (\ref{stoc_model}) in standard robust control form (refer to Fig. \ref{fig_cl_sys_unc}), we rewrite the system equation in slightly different form. We first define $u:=[\omega_G^\top\;\ P^\top]^\top$ and
\begin{align*}
du = & A u dt + b dt + \textstyle{\sum_{k=1}^{s}} \sigma \bar B_k ( \bar C_k u +  \bar G_k) d\xi_k
\end{align*}
where, 
\begin{small}
\begin{align*}
& A = \begin{bmatrix}
-M_G^{-1} D_G & -M_G^{-1} E_G \\ W^0 E_G^{\top} & -W^0E_L^{\top}D_L^{-1} E_L \end{bmatrix}, 
b = \begin{bmatrix}
M_G^{-1}P_G^m \\
 W^0 E_L^{\top} D_L^{-1} P_L^m
\end{bmatrix}, \\
& \bar B_k = \begin{bmatrix}
0 \\
e_k
\end{bmatrix}, 
\bar C_k = \begin{bmatrix}
(W^0 E_G^{\top})_k & -(W^0 E_L^{\top} D_L^{-1} E_L)_k
\end{bmatrix},  \\
& G_k = \begin{bmatrix}
( W^0 E_L^{\top} D_L^{-1} P_L^m)_k
\end{bmatrix},
\end{align*}
\end{small}
where $e_k\in \mathbb{R}^p$ is a vector of all zeros except for $1$ in the $k^{th}$ location. Chose $u^*$, such that, $A u^* + {b} = 0$ and define $ v = u - u^*$ to shift the equilibrium of the deterministic system to origin. Then, we have 
{\small
\begin{align}
d v = & A  v dt + \sum_{k=1}^{s} \sigma \bar B_k \bar C_k  v d\xi_k + \sum_{k=1}^{s} \sigma \bar B_k ( \bar C_k u^* + \bar G_k) d\xi_k \label{eq_full_stoch_sys}
\end{align} }
The matrix $A$ is singular and consists of non-zero null space. In the following, we perform a change of coordinates to separate the dynamics on and off the null space of $A$ matrix. 

Let ${\cal N}_s(A)$ and ${\cal R}_s(A)$ denotes the set of vectors which span the null space and range space of $A$. Then, define the transformation matrix, $V = \begin{bmatrix}
{\cal N}_s(A) & {\cal R}_s(A)
\end{bmatrix}$.
Using the transformation matrix, $V$, we define $[dx^\top\;\ dy^\top]^\top:=V^\top dv$. It can be shown, after the transformation, various transformed matrices has the following structure. 
\begin{small}
\begin{align*}
\begin{bmatrix}
0 & A_{yx} \\ 0 &  A_{xx}
\end{bmatrix} := & V^\top AV \\
\begin{bmatrix}
0 & (V^{\top}  \bar B_k \bar C_k V)_{yx} \\
0 & (V^{\top}  \bar B_k  \bar C_k V)_{xx}
\end{bmatrix} := & V^{\top} \bar B_k  \bar C_k V \\
\begin{bmatrix}
(V^{\top} \bar  B_k (\bar{C}_k u^*+\bar{G}_k))_{yx} \\
(V^{\top} \bar  B_k (\bar{C}_k u^*+\bar{G}_k))_{xx} 
\end{bmatrix}:= & V^{\top} \bar B_k (\bar {C}_k u^*+\bar {G}_k)
\end{align*}
\end{small}
where $A_{xx}\in \mathbb{R}^{n\times n}$, $A_{yx}\in\mathbb{R}^{n_l\times n} $. Using the fact that $\bar B_k$ and $\bar C_k$ are column vector and row vector respectively,  the matrix $\bar B_k \bar C_k$ is rank deficient. It is easy to show that $(V^{\top}  \bar B_k  \bar C_k V)_{xx}$ is also rank deficient and hence we write $(V^{\top}  \bar B_k  \bar C_k V)_{xx}\in \mathbb{R}^{n\times n}$ also as a product of column vector and row vector. In particular, we write, 
$B_kC_k:=(V^{\top}  \bar B_k  \bar C_k V)_{xx},$
where $B_k$ and $C_k$ are $n$ dimensional column vector and row vector respectively. Note that, the decomposition of the matrix as a product of two rank one matrices is not unique, but the final stability results are independent of the decomposition. Defining  $G_k:=(V^{\top} \bar  B_k (\bar{C}_k u^*+\bar{G}_k))_{xx} \in \mathbb{R}^n$ and ${\cal A}:=A_{xx}$ we write Eq. (\ref{eq_full_stoch_sys}) in the transformed coordinates as follows
\begin{align}
dy = & A_{yx} x dt + \textstyle{\sum_{k=1}^{s}} \sigma(V^{\top} \bar B_k \bar C_kV)_{yx} x d\xi_k \nonumber \\
& + \textstyle{\sum_{k=1}^{s}} \sigma(V^{\top} \bar B_k (\bar {C}_ku^*+\bar {G}_k))_{yx} d\xi_k, \label{eq_synchro} \\
dx = &  {\cal A} x dt + \textstyle{\sum_{k=1}^{s}} \sigma B_k C_k x d\xi_k + \textstyle{\sum_{k=1}^{s}} \sigma G_k d\xi_k, \label{eq_nonsynchro}
\end{align}
%
where $x \in \mathbb{R}^n$ and $y \in \mathbb{R}^{{n_g} - n}$. We notice that the $y$ dynamics is completely driven by $x$ dynamics and noise processes whereas, $x$ dynamics is not influenced by $y$ dynamics. Hence, the necessary condition for the stability of the above system of equations (\ref{eq_synchro})-(\ref{eq_nonsynchro}) is that, $x$ dynamics is stable. We define following notion of second moment bounded stability for system (\ref{eq_nonsynchro}).
\begin{definition}\label{def_smb}[Second Moment Bounded]
System (\ref{eq_nonsynchro}) is said to be second moment bounded if there exists a positive constant $\bar K$, such that
$\lim_{t\to \infty} E[x(t)^\top x(t)]\leq \bar K$ for all $x(0)\in \mathbb{R}^n$.
\end{definition}
Now, consider the stochastic power network without the additive noise term as follows.
\begin{align}
dx = & {\cal A} x dt + \textstyle{\sum_{i=1}^{s}} \sigma B_k C_k x d\xi_k. \label{eq_nonsynchro_noadd}
\end{align}
Following notion of mean square exponential stability can be introduced for (\ref{eq_nonsynchro_noadd}).
 \begin{definition}\label{def_mse}[Mean Square Exponentially Stable]
System (\ref{eq_nonsynchro_noadd}) is mean square exponentially stable, if there exist positive constants $K$ and $\beta$, such that
\[E[ x(t)^\top x(t) ]\leq K\exp^{-\beta t}E[{x(0)}^\top x(0)]\;\; \forall\; x(0)\in \mathbb{R}^n.\]
\end{definition}
The connection between the stability of systems given in Eqs. \eqref{eq_nonsynchro} and \eqref{eq_nonsynchro_noadd} is established in the following Lemma. 
\begin{lemma} The system (\ref{eq_nonsynchro_noadd}) is mean square exponentially stable if and only if system (\ref{eq_nonsynchro}) is second moment bounded. 
\label{lemma_equivalence}
\end{lemma}
\begin{proof}
Refer to Appendix for the proof. 
\end{proof}
Using Lemma \ref{lemma_equivalence}, it suffices to analyze the mean square exponential stability of system without additive noise given in Eq. \eqref{eq_nonsynchro_noadd}. 

In order to apply the results of mean square stability from \cite{Sai_arxiv,pushpak2015stability}, system \eqref{eq_nonsynchro_noadd} should be rewritten in the standard robust control form with a deterministic system in feedback with stochastic uncertainty (refer to Fig. \ref{fig_cl_sys_unc}). In doing so, we rewrite the stochastic power network model as feedback interconnection of the deterministic system and stochastic uncertainty. 
\begin{wrapfigure}[12]{l}{4.1 cm}
\begin{center}
{{\includegraphics[height=1.4 in]{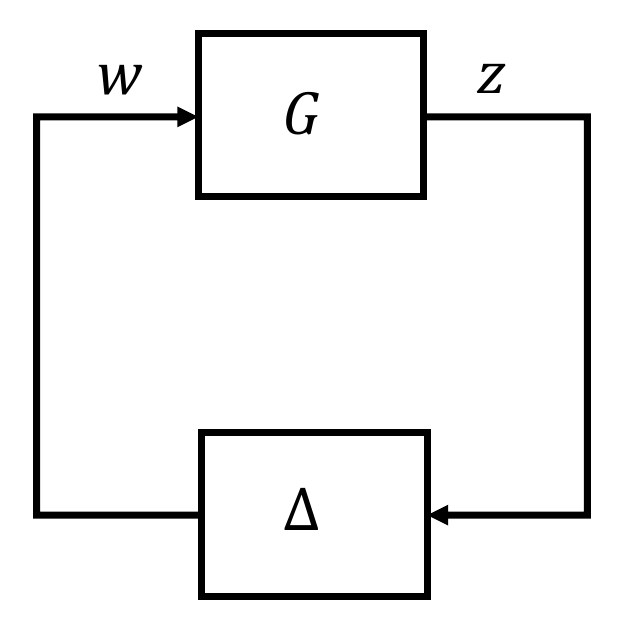}}}
\caption{Feedback interconnection of ${\cal G}$ \& $\Delta$, ${\cal F}({\cal G},\Delta).$}
\label{fig_cl_sys_unc}
\end{center}
\end{wrapfigure}
The deterministic part of the system is given by 
\begin{equation}
\mathcal{G}: \left\{
\begin{aligned}
\dot{x} = & \mathcal{A} x + \mathcal{B} w \\
z = & \mathcal{C} x 
\end{aligned}
\right. 
\label{eq_mean_sys}
\end{equation}
where the control and disturbance variables are respectively $z \in \mathbb{R}^{s}$ and $w \in \mathbb{R}^{s}$. The matrix $\mathcal{B}$ is formed by stacking  $B_k$'s in the columns and  $\mathcal{C}$ matrix is formed by stacking  $C_k$'s in the rows. 
This deterministic system, $\mathcal{G}$ now interacts with the stochastic uncertainty,  $\Delta$ and the interconnected system is denoted by $\mathcal{F}(\mathcal{G},\Delta)$. The stochastic power network model is written as feedback interconnection of deterministic and stochastic uncertainty as follows: 
%
\begin{equation}
\mathcal{F}(\mathcal{G},\Delta): \left\{
\begin{aligned}
\dot{x} = & \mathcal{A} x + \mathcal{B} w \\
z = & \mathcal{C} x \\
w = & \Delta z
\end{aligned} \right. 
\label{eq_full_sys}
\end{equation}
where the matrix, $\Delta$, is a diagonal matrix whose entries are $\sigma \rm{diag}( \frac{d\xi_1}{dt},\dots,\frac{d\xi_{s}}{dt})$. The stochastic uncertainty is interacting with the deterministic system through the control and disturbance variables. 
%
%
%
Clearly, for the stochastic system to be mean square stable, we require that the deterministic system is stable, i.e., ${\cal A}$ is Hurwitz. Note that, in the power system model, the matrix $\cal A$ is Hurwitz. Under the assumption that the system matrix $\cal A$ is Hurwitz, we have following necessary and sufficient condition for mean square stability.
\begin{theorem} \cite{Sai_arxiv} \label{theorem_stability} The feedback interconnection  ${\cal F}(\mathcal{G},\Delta)$ is mean square exponentially stable if and only if $\sigma \rho ( \hat{G}  )<1,$
where $\rho$ stands for the spectral radius of a matrix and 
\begin{eqnarray}{\hat G} := \begin{pmatrix}\parallel \mathcal{G}_{11}\parallel_2^2&\ldots &\parallel \mathcal{G}_{1s}\parallel_2^2\\\vdots&\ddots&\vdots\\\parallel \mathcal{G}_{s1}\parallel_2^2&\ldots&\parallel \mathcal{G}_{ss}\parallel_2^2 \end{pmatrix}\label{G_fornorm}.
\end{eqnarray}
The notation, $\parallel \mathcal{G}_{ij}\parallel_2$ is the ${\cal H}_2$ norm of the system $\mathcal{G}$ from disturbance input, $j$ and controlled output, $i$. 
\end{theorem}
Refer to \cite{Sai_arxiv} for the proof. The result of the above theorem can be used to compute critical value of $\sigma_*$ above which, the system is mean square unstable. In particular, the critical value, $\sigma_*$ is given by, $\sigma_*=\frac{1}{\rho(\hat G)}$. 

\section{IEEE 68 bus system}
\label{sec_68bus}
In this section, we consider the IEEE $68$ bus network to analyze the load-side primary frequency control with stochastic load voltages. 
IEEE $68$ bus New England/New York interconnection test system consists of $16$ generator buses and $52$ load buses. The single-line diagram of the $68$ bus test system is shown in Fig. \ref{fig_68bus}. This system contains induction motor loads, constant power loads, and controllable loads. The relevant data for this system is obtained from the data files of power system toolbox \cite{Chow_toolbox}. 
%

As discussed in Section \ref{sec_stochastic_model}, we include the renewables in the power network, and the uncertain, intermittent nature of renewables are modeled into the power network by considering the voltages to be stochastic. 
Changing the controllable loads involves a cost measured in the form of aggregate disutility of loads, and it has to be minimized.

In this $68$ bus system, there are $29$ induction motor loads which are sensitive to frequency, $35$ controllable loads and the remaining loads are uncontrollable frequency insensitive loads. Now, we replace few of the classical generators with renewable energy sources such as either solar or wind power. As there are no large moving parts at these renewables, the inertia values at these buses are relatively smaller, when compared to the classical generators \cite{Gautam_Vittal_wind_impact}. Further, integration of renewables into the power network increases the damping slightly \cite{Slootweg_wind_impact}. Therefore, we assume a relatively smaller value for inertia at renewables location and relatively bigger value for damping at those places. For the simulation purpose, we consider the renewable energy source at buses $54,55,56,60,63,64$ and $65$ replacing the generator buses. Now, the buses connecting the renewable buses are $6, 10, 19, 25, 32, 36$ and $52$ and hence, $s = 7$.

In the given data, the inertia values at generator buses lie between $1-5$, whereas, at the buses with renewable energy, we have considered it as $0.5$. Similarly, the damping values at generator buses are in the range of $0-5$, and we consider the damping at renewable energy buses to be $6$. The simulations results discussed below are consistent with the range of inertia values between $0.5-1$ and damping values between $5-6$. 
%
%

\begin{figure}[b]
\centering
\begin{minipage}{.22\textwidth}
  \centering
  \includegraphics[scale = 0.14]{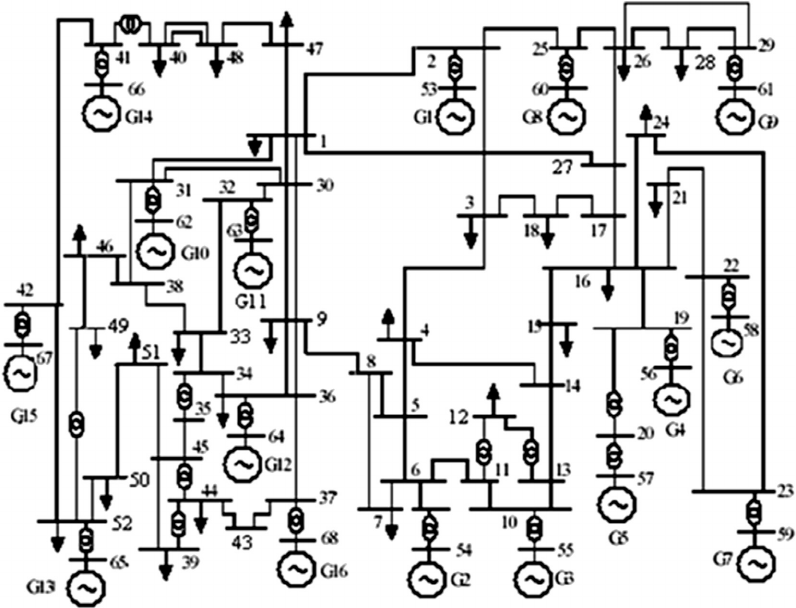}
  \caption{Single-line diagram of IEEE $68$ bus system}
  \label{fig_68bus}
\end{minipage}%
\begin{minipage}{.28\textwidth}
  \centering
  \includegraphics[scale = 0.28]{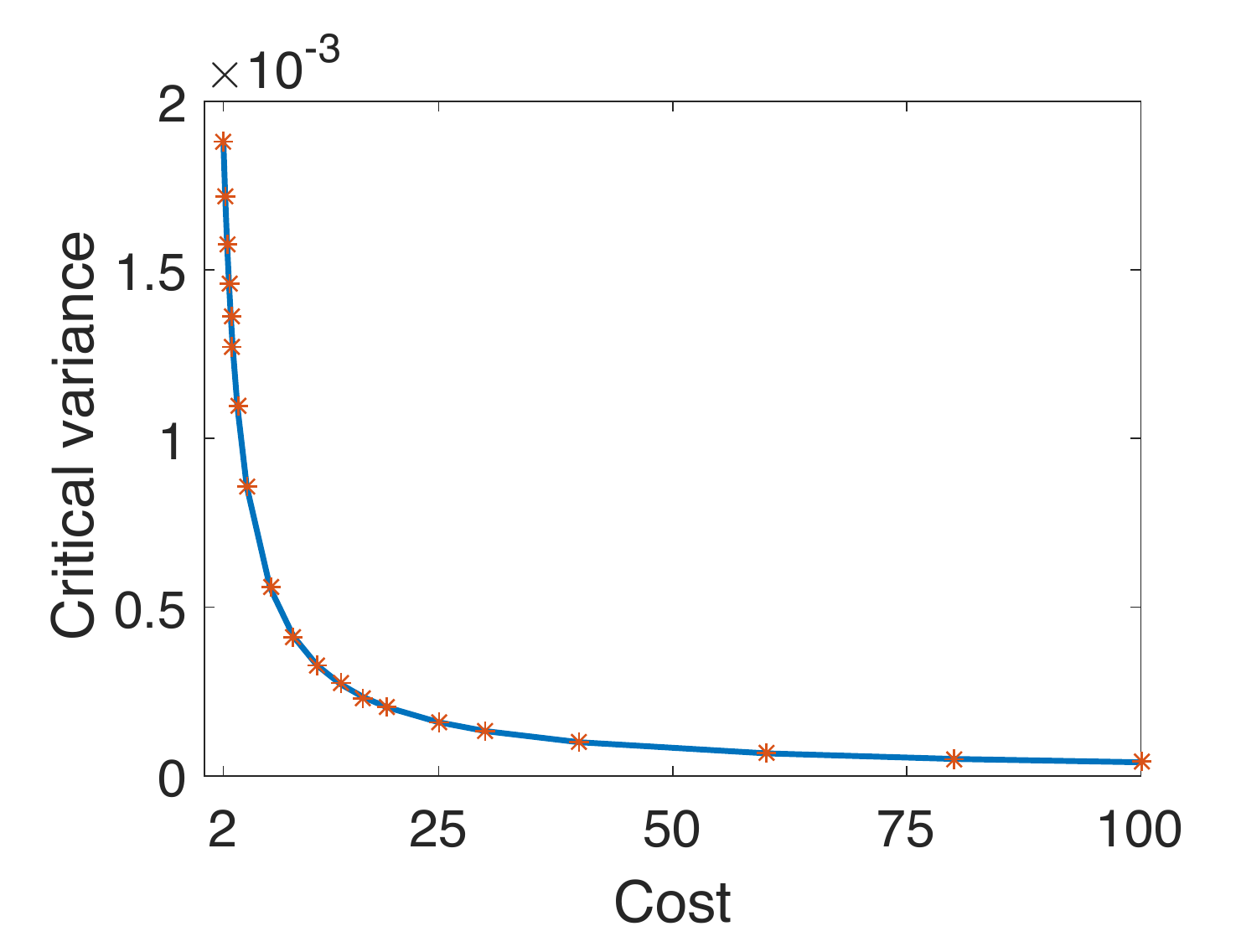}
  \caption{Variation of critical variance with cost}
  \label{fig_cost_vs_variance}
\end{minipage}
\end{figure}

Next, we analyze the effect of cost on controllable loads on the load-side primary frequency control with stochastic renewables. If the cost on controllable loads is high, then it is difficult to vary the controllable loads. Using the analytical framework discussed in section \ref{sec_stochastic_stability}, we identify the critical variance that can be tolerated in the voltages while maintaining the mean square exponential stability of power network with a decentralized controller.
The critical variance value, $\sigma_*^2$, is observed to be very small in the order of $10^{-3}$ with the maximum variance value of $1.9\times 10^{-3}$ which is obtained when the cost coefficient on the controllable load is equal to $\alpha=0.5$ (Refer to Assumption \ref{cost_assumption}). In Fig. \ref{fig_cost_vs_variance}, observe that, if the cost of controllable loads is further increased, the critical variance that can be tolerated by the stochastic power network reduces. It is important to notice that, for most of the cost values on controllable loads, the critical variance is very small. In generating Fig. \ref{fig_cost_vs_variance}, all the damping values at generator and load buses are kept constant.

Observe that, if the variance that can be tolerated by the system is small, then the system is on the verge of stability. This nature of the system can be seen, when we consider the stochastic voltages with a variance, $\sigma^2 > \sigma^2_*$, the frequencies grow out of bounds, and the power network becomes mean square unstable. This phenomenon is seen in Figs. \ref{fig_voltage_vs_time} and \ref{fig_frequency_vs_time_all_buses}. 
The stochastic voltage variation with respect to time is shown in Fig. \ref{fig_voltage_vs_time}. 
In Fig. \ref{fig_voltage_vs_time}, for the chosen $\sigma^2 > \sigma^2_*$, it is important to emphasize that although the voltages values lie within the safe operating limits of $0.95$ pu to $1.05$ pu, but the frequencies violate the operating limits as seen in Fig. \ref{fig_frequency_vs_time_all_buses}. This shows the fragility of the decentralized controller in the presence of renewables, as it is inadequate to regulate the frequency by means of controllable loads.  


\begin{figure}
\centering
\begin{minipage}{.22\textwidth}
  \centering
  \includegraphics[scale = 0.2]{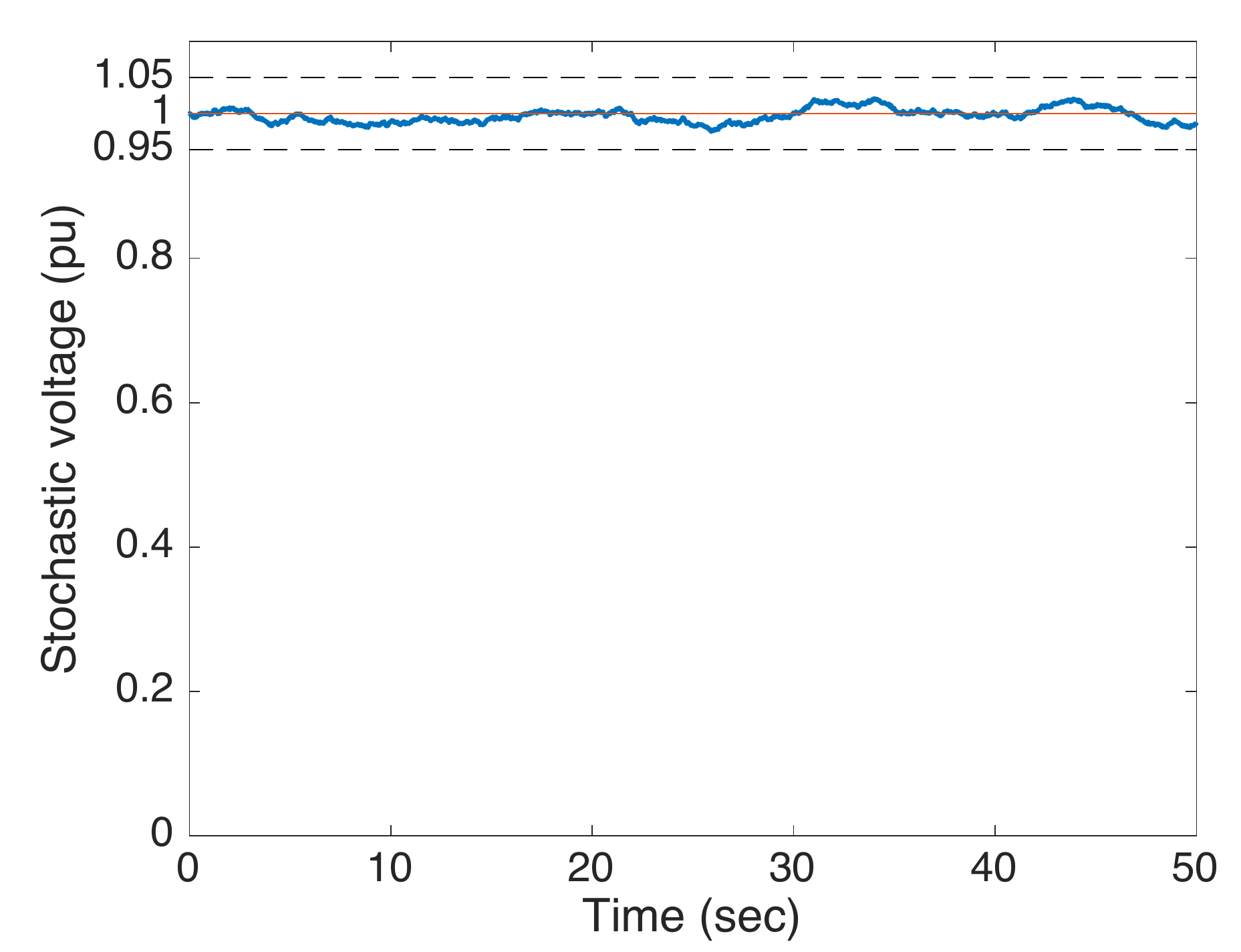}
  \caption{Voltage fluctuation at generator bus $56$}
  \label{fig_voltage_vs_time}
\end{minipage}%
\begin{minipage}{.28\textwidth}
  \centering
  \includegraphics[scale = 0.28]{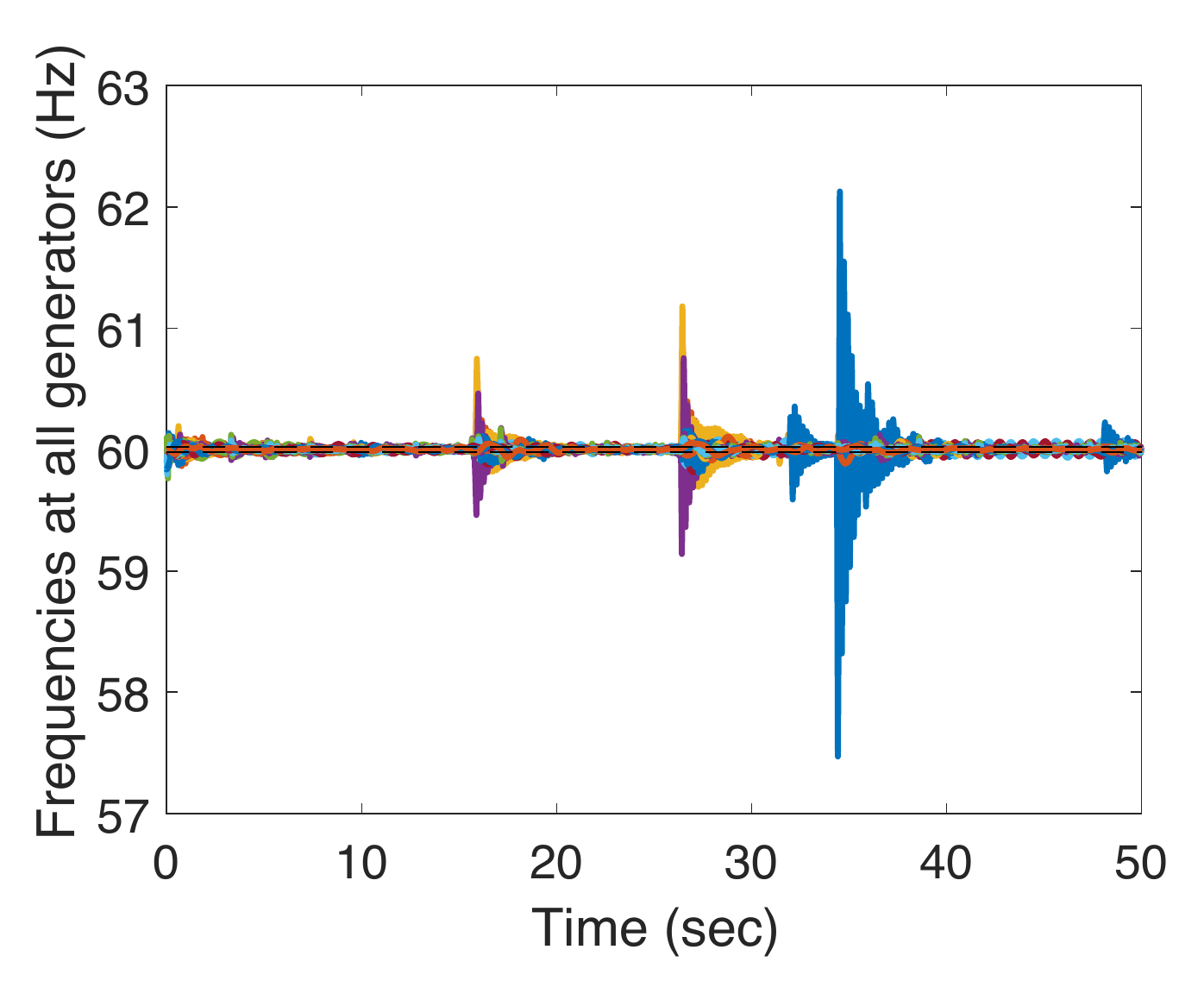}
  \caption{Mean square unstable behavior of frequencies at all generator buses.}
  \label{fig_frequency_vs_time_all_buses}
\end{minipage}
\end{figure}

%
%

Consider a step change in the power network. For this step change in power, the decentralized frequency controller is ineffective in controlling the controllable loads to regulate the frequency. In Fig. \ref{fig_frequency_with_disturbance}, initially the system was in stable operating condition with frequencies within the operating limits. We created a step change in power after $10$ seconds, and then the frequencies oscillate and go out of the operating range and continue to oscillate. This phenomenon is not desirable, as it has the impact to damage the power system equipment. 
%

Hence, to counteract the fragility of this decentralized controller, a modified robust distributive controller must be designed to regulate the frequency that can tolerate uncertainty in the renewables. 
%
%

\begin{figure}
\centering
\begin{minipage}{.22\textwidth}
  \centering
  \includegraphics[scale = 0.24]{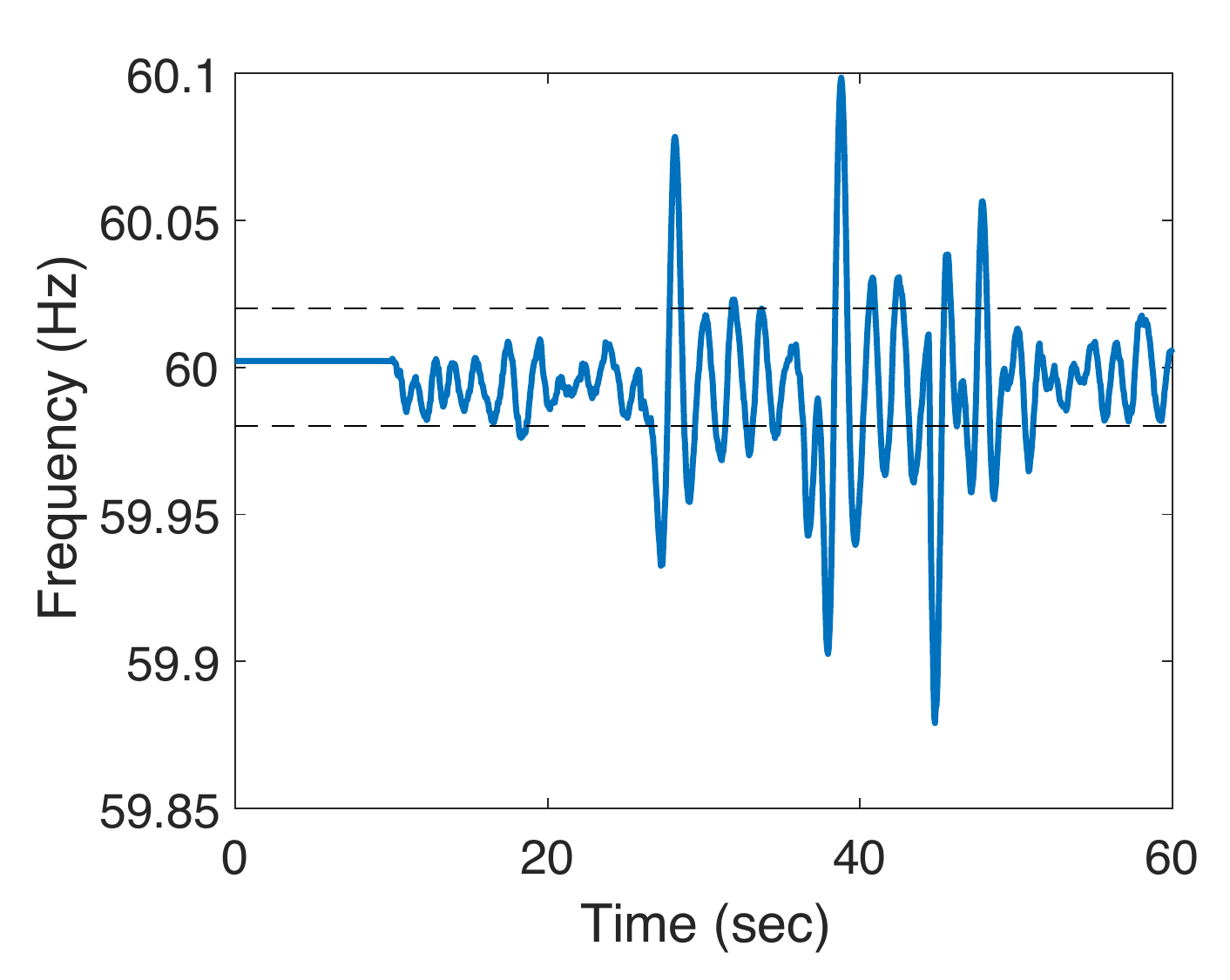}
  \caption{Frequency at  generator bus $53$ following a  step change in power.} 
  \label{fig_frequency_with_disturbance}
\end{minipage}%
\begin{minipage}{.28\textwidth}
  \centering
  \includegraphics[scale = 0.25]{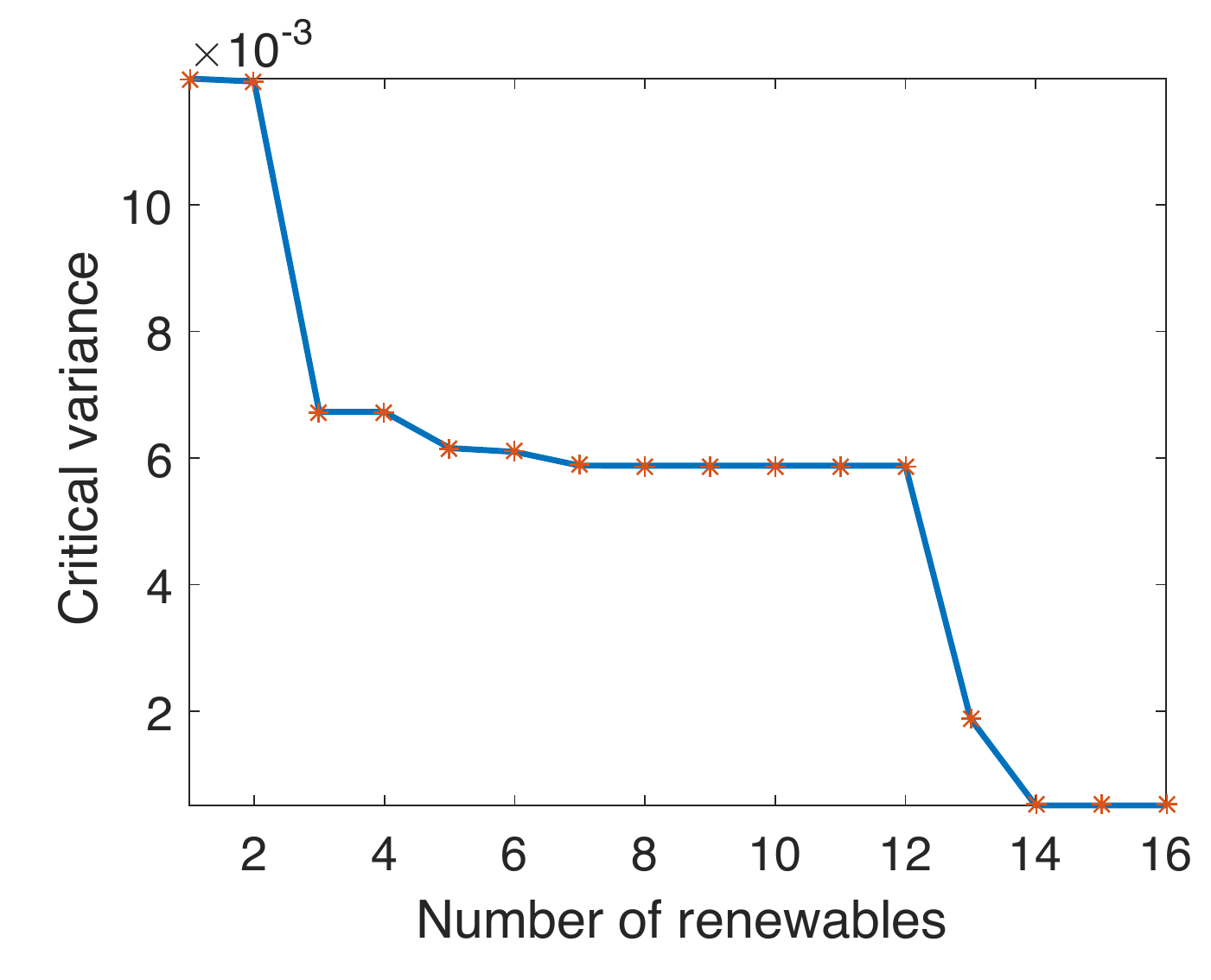}
  \caption{Critical variance with increase in number of renewables} 
  \label{fig_variance_vs_num_renewables}
\end{minipage}
\end{figure}

The higher penetration of renewables in the power network will make the decentralized frequency control algorithm more fragile. In particular, with the increase in the number of renewable energy resources, more bus voltages will become uncertain, and this has an adverse impact on the frequency regulation. Fig. \ref{fig_variance_vs_num_renewables} shows the effect of increasing the penetration of renewables in the power network. We notice, with the increase in penetration (i.e., with an increase in the value of $s$) the critical variance that can be tolerated by the system decreases. Note that, this figure will change based upon which locations in the network are chosen for renewables. However, the trend of decrease in the value of critical variance with the increase in the number of renewables will continue to hold true. 



\section{Conclusion}
\label{sec_conclusion}
We showed that the decentralized load-side frequency regulation algorithm is fragile to stochastic parametric uncertainty in a power network. The presence of stochastic uncertainty is motivated through uncertainty in renewables. We show that the decentralized algorithm becomes more fragile with the increase in the cost of the controllable loads and also with the increased degree of penetration of renewable resources. System theoretic-based analysis and synthesis framework developed for the stochastic networked system in \cite{Sai_arxiv,sai_acc2016} is used to prove the main results. Our future research efforts will be focused on the design of distributed frequency regulation algorithm robust to stochastic uncertainty in power network using the synthesis framework developed in \cite{Sai_arxiv}.



 \bibliographystyle{ieeetran}
 \bibliography{references}
 

 \section*{appendix}
Here, first we recall from \cite{Sai_arxiv}, the covariance propagation equation for the systems given in \eqref{eq_nonsynchro_noadd} and \eqref{eq_nonsynchro}. Let $Q(t)$ and $\bar Q(t)$ be the covariance matrices corresponding to \eqref{eq_nonsynchro_noadd} and \eqref{eq_nonsynchro}. Then, they satisfy the following matrix differential equations (MDE's). 
\begin{small}
\begin{align}
\dot{{Q}}(t)=& {Q}(t){\cal A}^\top + {\cal A} {Q}(t) + \sum_{k=1}^p \sigma_k^2 B_k {Q}(t)B_k^\top, \label{cov_mde} \\
\dot {\bar Q}(t)=& \bar Q(t) {\cal A}^\top+{\cal A} \bar Q(t)+\sum_{k=1}^p \sigma_k^2 B_k \bar Q(t) B_k^\top + \sum_{k=1}^p G_kG_k^\top\nonumber \\  
& + \sum_{k=1}^p \sigma_k G_k \mu(t)^{\top} B_k^{\top} + \sum_{k=1}^p \sigma_k B_k \mu(t) G_k^{\top}. \label{cov_mde_add} 
\end{align}
\end{small}
The following equation shows the mean propagation equation for the system with additive noise given in \eqref{eq_nonsynchro}. 
\begin{align}
\dot \mu(t) = & {\cal A} \mu(t)
\label{eq_mean_propagation}
\end{align}
Next, we present the proof of Lemma \ref{lemma_equivalence}. 

\begin{proof}[Lemma \ref{lemma_equivalence}]
Using the operator $\phi$, that transforms a matrix into a vector as defined in \cite{Costa_book}[Chapter 2], the MDE's given in Eq. \eqref{cov_mde} and Eq. \eqref{cov_mde_add} are written as linear differential equations as  given below. 
\begin{eqnarray}
\dot{q} = & \mathscr{A} q,
\label{cov_lde} \\
\dot{\bar q} = & \mathscr{A} \bar q + \mathscr{B},
\label{cov_lde_add}
\end{eqnarray}
where $q = \phi( Q), \bar q = \phi(\bar Q)$, $\mathscr{B} = \sum_{k=1}^p ((G_k \otimes G_k) + ((\sigma_k G_k \mu^{\top}) \otimes B_k) + (B_k \otimes (\sigma_k \mu^{\top} G_k)))\phi(I) \in \mathbb{R}^{n^2} \text{and}$ 
 $\mathscr{A} = {\cal A} \oplus {\cal A} + \sum_{k=1}^p \sigma_k^2 (B_k \otimes B_k)\in\mathbb{R}^{n^2\times n^2},$
where $I$ is the identity matrix of size $n \times n$ and $\otimes$ denotes the Kronecker product, $\oplus$ is the Kronecker sum.

\textit{Necessity}:
The mean square exponential stability of system \eqref{eq_nonsynchro_noadd} yields stability of system \eqref{cov_lde}, that is, $\mathscr{A}$ is Hurwitz. Since  $\mathscr{A}$ is Hurwitz, the steady state value of $\bar q$ is given by 
$\lim_{t \to \infty} \bar q(t) =    \lim_{t \to \infty}  \phi(\bar Q(t)) =  -\mathscr{A}^{-1} (\sum_{k=1}^p G_k \otimes G_k) \phi(I).$
Now, taking the inverse $ \phi$ operator, we obtain,
$ \lim_{t \to \infty} E[x(t)x(t)^{\top}] =  - \phi^{-1}(\mathscr{A}^{-1} (\sum_{k=1}^p G_k \otimes G_k) \phi(I)),$
which is finite. Further, the necessary condition for $\mathscr{A}$ to be Hurwitz is ${\cal A}$ being Hurwitz. This implies that the mean propagation system of \eqref{eq_nonsynchro}, shown in Eq. \eqref{eq_mean_propagation} has a stable evolution. Therefore, system \eqref{eq_nonsynchro} is second moment bounded.

\textit{Sufficiency:} If system \eqref{eq_nonsynchro} is second moment stable, then
$ \lim_{t \to \infty} \bar Q(t) $ is a finite value and the mean system \eqref{eq_mean_propagation} has a stable evolution. Taking the operator, it can be alternately written as,
$\lim_{t \to \infty}  \phi(\bar Q(t))  =\lim_{t\to \infty} e^{\mathscr{A}t}  \phi(\bar Q(0)) - \mathscr{A}^{-1} (1-e^{\mathscr{A}t}) (\sum_{k=1}^p G_k \otimes G_k) \phi(I) + e^{\mathscr{A}t} (\sum_{k=1}^p (\sigma_k G_k \mu^{\top}) \otimes B_k) \phi(I) + e^{\mathscr{A}t} (\sum_{k=1}^p B_k \otimes (\sigma_k \mu^{\top} G_k)) \phi(I).$
The limit on the right-hand side is finite, if and only if $\mathscr A$ is Hurwitz. If $\mathscr{A}$ is Hurwitz, then the system (\ref{eq_nonsynchro_noadd}) is mean square exponentially stable.
\end{proof}
\end{document}